\documentclass{amsart}
\usepackage[margin=1.2in]{geometry}
\usepackage{amssymb}
\usepackage{amscd}
\usepackage[all]{xy}
\usepackage{bbm}
\usepackage{mathrsfs}
\usepackage{enumerate}
\usepackage{tikz}
\usepackage{stmaryrd}
\usepackage{etoolbox}
\usepackage{array}

\newcommand{\R}{\mathbb{R}}

\newcommand{\inv}{^{-1}}

\newcommand{\eps}{\varepsilon}

\newcommand{\del}{\nabla}

\newcommand{\lap}{\Delta}

\newcommand{\bd}{\partial}

\newcommand{\cc}{\mathfrak c}
\renewcommand{\div}{\operatorname{div}}

\newcommand{\grad}{\del}
\newcommand{\vol}{\operatorname{Vol}}

\newcommand{\area}{\operatorname{Area}}

\newcommand{\ric}{\operatorname{Ric}}

\newcommand{\cp}{\operatorname{cap}}

\newcommand{\m}{m_{\operatorname{ADM}}}

\theoremstyle{plain}
\newtheorem{theorem}{Theorem}
\newtheorem{corollary}[theorem]{Corollary}
\newtheorem{prop}[theorem]{Proposition}
\newtheorem{lem}[theorem]{Lemma}

\theoremstyle{definition}
\newtheorem{defn}[theorem]{Definition}
\newtheorem{rem}[theorem]{Remark}
\newtheorem{ex}[theorem]{Example}

\begin{document}
\title{Mass, Conformal Capacity, and the Volumetric Penrose Inequality} 
\author{Liam Mazurowski}
\author{Xuan Yao}

\begin{abstract} Let $\Omega$ be a smooth, bounded subset of $\mathbb{R}^3$ diffeomorphic to a ball. Consider $M = \mathbb{R}^3 \setminus \Omega$ equipped with an asymptotically flat metric $g = f^4 g_{\text{euc}}$, where  $f\to 1$ at infinity.  Assume that $g$ has non-negative scalar curvature and that $\Sigma = \partial M$ is a minimal 2-sphere in the $g$ metric. We prove a sharp inequality relating the ADM mass of $M$ with the conformal capacity of $\Omega$. As a corollary, we deduce a sharp lower bound for the ADM mass of $M$ in terms of the Euclidean volume of $\Omega$.  We also prove a stability type result for this ``volumetric Penrose inequality.'' The proofs are based on a monotonicity formula holding along the level sets of a 3-harmonic function. 
\end{abstract}

\maketitle

\section{Introduction} 

Let $(M^3,g)$ be an asymptotically flat manifold with one end. In intuitive terms, this means that the end of $M$ is diffeomorphic to $\R^3$ minus a ball, and the metric $g$ converges to the Euclidean metric at infinity.  Asymptotically flat manifolds play an important role in the study of general relativity. Associated to any asymptotically flat manifold $M$ is a quantity called the ADM mass, which measures the total mass of the gravitational system modeled by $M$.  Many important results in mathematical relativity center around proving lower bounds for the ADM mass.  In particular, suppose that $M$ has non-negative scalar curvature and that $\bd M$ is an outermost minimal surface.  Then the positive mass theorem says that $\m \ge 0$,  the Penrose inequality says that 
\[
\m \ge \sqrt{\frac{\area(\bd M)}{16\pi}},
\]
and the mass to $p$-capacity inequalities give a sharp lower bound for $\m$ in terms of the $p$-capacity of $\bd M$ for $1 < p < 3$. See the next subsection for an  extended discussion. 

Now consider the special case where $\Omega$ is a smooth, bounded domain in $\R^3$ diffeomorphic to a ball, and  $M = \R^3 \setminus \Omega$, and $g = f^4 g_{\text{euc}}$ where $f > 0$ is a smooth function with $f\to 1$ at infinity.
 The purpose of this paper is to prove a mass to capacity inequality for such manifolds $M$ in the conformally invariant case $p = 3$. Given $\Omega \subset \R^3$, we let $\cc(\Omega)$ denote the conformal capacity of $\Omega$. To the authors' knowledge, the conformal capacity for subsets of $\R^n$ with $n\ge 3$ was first introduced in \cite{colesanti2005brunn}. It is a natural generalization of the well-known and thoroughly studied logarithmic capacity for subsets of $\R^2$. See Definition \ref{conformal-capacity} for the precise definition. 
\begin{theorem}
\label{main}
Let $\Omega\subset \R^3$ be a smooth, bounded domain diffeomorphic to a ball.  Consider an asymptotically flat manifold $(M,g)$ where $M = \R^3\setminus \Omega$ and $g = f^4 g_{\text{euc}}$ for a smooth function $f > 0$ with $f \to 1$ at infinity. Assume that $M$ has non-negative scalar curvature and that $\Sigma = \bd M$ is a minimal 2-sphere in the $g$ metric. Then \[
\m \ge 2\cc(\Omega).
\]
Equality holds if and only if $M$ is Schwarzschild.  
\end{theorem}

As a corollary, we can deduce the following volumetric Penrose inequality relating the ADM mass of a conformally flat, asymptotically flat manifold to the Euclidean volume of $\Omega$.    

\begin{corollary}
\label{vpi}
Let $\Omega\subset \R^3$ be a smooth, bounded domain diffeomorphic to a ball.  Consider an asymptotically flat manifold $(M,g)$ where $M = \R^3\setminus \Omega$ and $g = f^4 g_{\text{euc}}$ for a smooth function $f > 0$ with $f \to 1$ at infinity.  Assume that $M$ has non-negative scalar curvature and that $\Sigma = \bd M$ is a minimal 2-sphere in the $g$ metric. Then \[
\m \ge 2\left(\frac{3}{4\pi}\vol(\Omega)\right)^{1/3}. 
\]
Equality holds if and only if $M$ is Schwarzschild.  
\end{corollary}

Volumetric Penrose inequalities of this nature were previously obtained by Schwartz \cite{schwartz2011volumetric} and Freire-Schwartz \cite{freire2014mass}. Compared with Corollary \ref{vpi}, their results hold in all dimensions and do not need $\Omega$ to be diffeomorphic to a ball. However, their results require extra assumptions about the conformal factor $f$ and the Euclidean geometry of $\Sigma$. 

We can also prove the following stability result for the volumetric Penrose inequality. 
Recall that the Fraenkel asymmetry of a set $\Omega \subset \R^3$ is defined by 
\[
\alpha(\Omega) = \inf\left\{ \frac{\vol(\Omega \operatorname{\Delta} B)}{\vol(\Omega)}: \text{$B$ is a ball with the same volume as $\Omega$}\right\}. 
\]
Thus $\alpha(\Omega)$ measures the scale invariant $L^1$ distance between $\Omega$ and a ball. 
The following theorem shows that the excess in the volumetric Penrose inequality is bounded from below by a constant times the square of the Fraenkel asymmetry. This resembles the conclusion of known stability theorems for the isoperimetric inequality \cite{fusco2008sharp} and the isocapacitary inequality \cite{de2021sharp},\cite{mukoseeva2023sharp}. 

\begin{theorem}
\label{stability}
There are constants $\eta_0 > 0$ and $C > 0$ such that the following holds.   Let $\Omega\subset \R^3$ be a smooth, bounded domain diffeomorphic to a ball.  Consider an asymptotically flat manifold $(M,g)$ where $M = \R^3\setminus \Omega$ and $g = f^4 g_{\text{euc}}$ for a smooth function $f > 0$ with $f \to 1$ at infinity.  Assume that $M$ has non-negative scalar curvature and that $\Sigma = \bd M$ is a minimal 2-sphere in the $g$ metric.  If 
 \[
 \frac{\m}{2\left(\frac{3}{4\pi}\vol(\Omega)\right)^{\frac 1 3}} \le 1+\eta
 \]
 for some $0 < \eta < \eta_0$, then $\alpha(\Omega) \le C\sqrt{\eta}$. 
 \end{theorem}

\subsection{Background and Discussion} Let $(M^3,g)$ be an asymptotically flat manifold. Assume that $M$ has non-negative scalar curvature and that $\Sigma = \bd M$ is an outermost minimal surface. The positive mass theorem asserts that $\m\ge 0$, with equality if and only if $M$ is Euclidean $\R^3$. The positive mass theorem was first proven by Schoen and Yau \cite{schoen1979proof} in 1979 using minimal surface techniques. Later, Witten \cite{witten1981new} gave another proof using spinor methods. 
As a test of the cosmic censorship hypothesis, Penrose \cite{penrose1973naked} proposed an inequality which, in the Riemannian setting, asserts that the ADM mass should be bounded below in terms of the area of the boundary via 
\[
\m \ge \sqrt{\frac{\area(\Sigma)}{16\pi}}. 
\]
This {\it Riemannian Penrose inequality} was confirmed by Huisken and Ilmanen \cite{huisken2001inverse} in the case where $\Sigma$ consists of a single 2-sphere. Their proof is based on weak inverse mean curvature flow.  Bray \cite{bray2001proof} gave an independent proof that does not require the boundary of $M$ to be connected.

In the 1980s, it was already known to the Polish school of physicists that $p$-harmonic functions could be a useful tool for proving the positive mass theorem; see the works of Kijowski \cite{kijowski1984unconstrained} and Chru{\'s}ciel \cite{chrusciel1986kijowski}. Recently, there has been considerable interest in obtaining rigorous lower bounds for the ADM mass using $p$-harmonic functions.  These lower bounds are formulated in terms of the $p$-capacity of $\Sigma$. By definition, for $1 < p < 3$, the $p$-capacity of $\Sigma$ is 
\[
\operatorname{cap}_p(\Sigma) = \inf\left\{\int_M \vert \grad w\vert^p\, dv: w\in C^\infty_c(M) \text{ and } w\ge 1 \text{ on } \bd M\right\}.
\]
It is known that the infimum is achieved by a function $u\in W^{1,p}(M)$ which solves the PDE 
\[
\begin{cases}
\lap_p u = 0, &\text{in M}\\
u = 1, &\text{on } \Sigma\\
u\to 0, &\text{at infinity}. 
\end{cases}
\]
The function $u$ has $C^{1,\alpha}$ regularity and is smooth away from its critical points.  The $p$-capacity can also be computed in terms of $u$ via 
\[
\operatorname{cap}_p(\Sigma) = \int_M \vert \grad u\vert^p\, dv = \int_{\{u=t\}} \vert \grad u\vert^{p-1}\, da,
\]
where $\{u=t\}$ is any regular level set of $u$. We will call $u$ the $p$-capacitary potential associated to $\Sigma$.  

In the harmonic case $p=2$, Bray \cite{bray2001proof} proved a mass-capacity inequality 
\begin{equation}
\label{2-cap}
\m \ge \frac{\operatorname{cap}_2(\Sigma)}{4\pi}
\end{equation}
by using the positive mass theorem and a reflection argument. Equality holds if and only if $M$ is Schwarzschild.  
In \cite{bray2008capacity}, Bray and Miao gave another proof based on weak inverse mean curvature flow. Later, Miao \cite{miao2023mass} and Oronzio \cite{oronzio2022adm} gave independent proofs of (\ref{2-cap}) that rely only facts about harmonic functions. 

For $1 < p < 3$, Xiao \cite{xiao2016p} adapted the method of Bray and Miao \cite{bray2008capacity} to prove the mass to $p$-capacity inequality 
\begin{equation}
\label{p-cap}
\m \ge 2\left(\frac{\operatorname{cap}_p(\Sigma)}{K(p)}\right)^{\frac{1}{3-p}}, 
\end{equation}
where $K(p)$ is an explicitly computable constant.\footnote{In fact, $K(p)$ is the $p$-capacity of the horizon in mass 2 Schwarzschild.} Equality holds if and only if $M$ is Schwarzschild.  This recovers the Penrose inequality in the limit $p\to 1$.  Agostiniani, Mantegazza, Mazzieri, and Oronzio \cite{agostiniani2022riemannian} proved a version of (\ref{p-cap}) with non-sharp constant using only facts about $p$-harmonic functions and non-linear potential theory. Their weaker inequality is still enough to recover the Penrose inequality in the $p\to 1$ limit. Later, the authors \cite{mazurowski2023monotone} gave a proof of the sharp inequality (\ref{p-cap}) using only facts about $p$-harmonic functions; also see \cite{xia2024new} for an independent proof along the same lines. For more applications of $p$-harmonic functions ($1<p<3$) in the study of scalar curvature, see the references \cite{chan2024monotonicity} and \cite{hirsch2024monotone}. 

The goal of this paper is to study the mass capacity inequality in the conformal case $p = 3$. We will restrict our attention to conformally flat, asymptotically flat manifolds $(M,g)$ where $M = \R^3 \setminus \Omega$ and $g = f^4 g_{\text{euc}}$ for some smooth $f > 0$ satisfying  $f\to 1$ at infinity.   We will always assume that $M$ has non-negative scalar curvature and that $\Sigma = \bd M$ is minimal.  It is straightforward to check that 
\[
\operatorname{cap}_3(\Sigma) = \inf\left\{\int_{\R^3 \setminus \Omega} \vert \grad w\vert^3 \, dv: w\in C^\infty_c(\R^3\setminus \Omega) \text{ and } w\ge 1 \text{ on } \Sigma\right\} = 0.  
\]
Hence this does not provide a useful notion of capacity when $p = 3$. 

However, a natural conformal capacity can be defined as follows. Fix $1 < p < 3$. Let $u$ be the $p$-capacitary potential associated to $\Sigma$. Then it is known that $u$ has an asymptotic expansion 
\[
u = a{\vert x\vert^{\frac {p-3}{p-1}}} + o(\vert x\vert^{\frac {p-3}{p-1}}).
\]
In particular, $u$ decays like the fundamental solution of the $p$-Laplacian near infinity, and the $p$-capacity can be computed in terms of the constant $a$ in this expansion via 
\[
\operatorname{cap}_p(\Omega) =  \lim_{t\to \infty} \int_{\{u=t\}} \vert \grad u\vert^{p-1}\, da =  4\pi \left(\frac{3-p}{p-1}\right)^{p-1}  a^{p-1}.
\] 
In the conformal case $p=3$, the fundamental solution of the 3-Laplacian no longer decays near infinity, but rather grows like $\log \vert x\vert$. The conformal capacity can then be defined in terms of the asymptotic expansion of the 3-harmonic function $u$ which is 0 on $\Sigma$ and grows like $\log \vert x\vert$ at infinity. More precisely, it is known that there is a unique weak solution $u$ to the PDE 
\[
\begin{cases}
\lap_3 u = 0, &\text{in } M\\
u = 0, &\text{on } \Sigma \\
\frac{u(x)}{\log \vert x\vert} \to 1, &\text{as } \vert x\vert \to \infty. 
\end{cases}
\]
Moreover, $u$ has an expansion $u(x) = \log \vert x\vert + a + o(1)$ near infinity. The conformal capacity of $\Omega$ is defined to be 
\[
\cc(\Omega) = e^{-a}.
\] 
A similar definition of conformal capacity (typically called {\it logarithmic capacity}) is well-known for subsets of $\R^2$, where it turns out to be equivalent to other conformal invariants like the transfinite diameter. To the authors' knowledge, this definition of conformal capacity for subsets of $\R^n$ with $n\ge 3$ first appeared in \cite{colesanti2005brunn}. Further properties of the conformal capacity in dimension $n\ge 3$ have been studied by Xiao \cite{xiao2013towards},\cite{xiao2018exploiting},\cite{xiao2020geometrical}. 

Our main result Theorem \ref{main} gives a lower bound for the mass of a conformally flat, asymptotically flat manifold in terms of the conformal capacity. The proof is based only on the properties of 3-harmonic functions. This inequality can be thought of as the natural mass to $p$-capacity inequality in the limiting case $p = 3$.  It seems an interesting problem to define the conformal capacity for general asymptotically flat manifolds which are not conformally flat. We give some further discussion of this problem in Appendix \ref{general}.

It is an easy corollary of the mass to conformal capacity inequality that the mass of a conformally flat, asymptotically flat manifold is bounded below in terms of the Euclidean volume of $\Omega$.  Regarding this, 
we note that Schwartz \cite{schwartz2011volumetric} has previously obtained the following volumetric Penrose inequality with sub-optimal constant. 

\begin{theorem}[Schwartz \cite{schwartz2011volumetric}]
Assume that $M = \R^n\setminus \Omega$. Suppose that $g = f^{\frac{4}{n-2}} g_{\operatorname{euc}}$ is an asymptotically flat metric on $M$ with non-negative scalar curvature, where $f$ is normalized so that $f\to 1$ at infinity. Let $\Sigma = \bd \Omega$ and assume that $\Sigma$ is Euclidean mean convex and $g$-minimal. Then 
\[
\m > \left(\frac{\vol(\Omega)}{\omega_n}\right)^{\frac{n-2}{n}},
\]
where $\omega_n$ is the volume of the unit ball in $\R^n$. 
\end{theorem}

Note that equality never holds in this theorem. In fact, the right hand side is a factor of 2 away from being sharp. Later, Freire and Schwartz \cite{freire2014mass}  obtained an optimal volumetric Penrose inequality when $\Sigma$ is Euclidean mean convex and the conformal factor $f$ satisfies certain additional assumptions.

\begin{theorem}[Freire and Schwartz \cite{freire2014mass}]
Assume that $M = \R^n\setminus \Omega$. Suppose that $g = f^{\frac{4}{n-2}} g_{\operatorname{euc}}$ is an asymptotically flat metric on $M$ with non-negative scalar curvature, where $f$ is normalized so that $f\to 1$ at infinity.  Let $\Sigma = \bd \Omega$ and define $\alpha = \min_\Sigma f$ and $\omega = \max_\Sigma f - \min_\Sigma f$. Assume that 
\begin{itemize}
\item[(i)] $\Sigma$ is Euclidean mean convex and $g$-minimal,
\item[(ii)] either $\alpha \ge 2$, or $\alpha \ge 1 + \frac{\sqrt 2}{2}$ and $\omega < 1 - \frac{\alpha}{2}$, or $\alpha < 1 + \frac{\sqrt 2}{2}$ and $\omega < 2-\alpha - \frac{1}{2\alpha}$. 
\end{itemize}
Then one has 
\[
\m \ge 2\left(\frac{\vol(\Omega)}{\omega_n}\right)^{\frac{n-2}{n}},
\]
where $\omega_n$ is the volume of the unit ball in $\R^n$. If equality holds, then $M$ is Schwarzschild. 
\end{theorem} 

By contrast, our Corollary \ref{vpi} obtains the sharp constant with no mean convexity assumption on $\Sigma$. It also dispenses with the extra assumption (ii) on the boundary behavior of the conformal factor $f$.  However, it is restricted to dimension 3 and requires that $\Sigma$ is a 2-sphere.  

The basic strategy in \cite{schwartz2011volumetric} and \cite{freire2014mass} is to first bound the mass from below by the 2-capacity in the $g$ metric, then to relate the 2-capacity in the $g$ metric with the $2$-capacity in the Euclidean metric, and finally to use the isocapacitary inequality to bound the Euclidean $2$-capacity from below by the volume. It is the middle step comparing the $2$-capacity in the $g$ metric to the $2$-capacity in the Euclidean metric which requires extra assumptions. Using the conformal capacity obviates the need for this extra middle step, since the conformal capacity is the same in the $g$-metric and the Euclidean metric. 

\begin{rem}
We want to emphasize that the volumetric Penrose inequality is not a straightforward consequence of the usual Penrose inequality. Indeed, applying the usual Penrose inequality and the isoperimetric inequality one can deduce that  
\[
\m \ge \left(\frac{\area_g(\Sigma)}{16\pi}\right)^{\frac 1 2} \ge \frac{(\min_{\Sigma} f)^2}{2} \left(\frac{\area_{\text{euc}}(\Sigma)}{4\pi}\right)^{\frac 1 2} \ge \frac{(\min_\Sigma f)^2}{2} \left(\frac{3}{4\pi}\vol(\Omega)\right)^{\frac{1}{3}}, 
\]
which only yields the volumetric Penrose inequality if $\min_\Sigma f \ge 2$. 
\end{rem}

Finally, we also prove a stability type theorem for the volumetric Penrose inequality. Theorem 3 shows that if $(M,g)$ is close to achieving equality in the volumetric Penrose inequality, then $\Omega$ must be close to a ball in a scale-invariant $L^1$ sense.
More precisely, we show that the excess in the volumetric Penrose inequality is bounded below by the square of the Fraenkel asymmetry. Similar stability theorems are known for the isoperimetric inequality \cite{fusco2008sharp} and the isocapacitary inequality for $p$ strictly between 1 and the dimension \cite{de2021sharp},\cite{mukoseeva2023sharp}. In all cases, the relevant excess is bounded below by the square of the Fraenkel asymmetry. We note that our method of proof is completely different from the methods in \cite{de2021sharp},\cite{fusco2008sharp},\cite{mukoseeva2023sharp}. 

\subsection{Structure of Paper and Outline of Proof}

The proof of Theorem \ref{main} is based on a monotonicity formula holding along the level sets of a 3-harmonic function. 
Let $\Omega\subset \R^3$ be a smooth,  bounded domain diffeomorphic to a ball so that $\Sigma = \bd \Omega$ is a 2-sphere. Let $(M,g)$ be a conformally flat, asymptotically flat manifold with $M = \R^3\setminus \Omega$ and $g = f^4 g_{\text{euc}}$ with $f\to 1$ at infinity. Assume that $M$ has non-negative scalar curvature and that $\Sigma = \bd M$ is minimal. Let $u$ be the conformal capacity function associated to $\Omega$. 

Assume for simplicity that $u$ is smooth with no critical points. In this case, each level set $\Sigma_t = \{u=t\}$ is a smooth 2-sphere. 
In Section \ref{smooth}, we derive a first order differential inequality satisfied by the function 
\[
U(t) = \int_{\Sigma_t} H_g\vert \grad^g u\vert\, da_g.
\]
More precisely, we show that 
\[
U'(t) \le 4\pi \left(1- \frac{1}{16\pi}\int_{\Sigma_t} H_g^2\, da_g\right) \le 4\pi - \frac{U(t)^2}{16\pi}. 
\]
This inequality implies that 
\begin{equation}
\label{eq-U}
U(t) \le U_s(t)
\end{equation}
where 
$
U_s(t)
$
is the $U$ function computed in the model case of Schwarzschild. In Section \ref{critical}, we rigorously justify the monotonicity even in the presence of critical points using an argument from \cite{agostiniani2022riemannian}. 

Next we estimate the asymptotics of the function $U$ as $t\to \infty$. Heuristic calculations similar to those in \cite{miao2023mass} suggest that 
\begin{equation}
\label{asy1}
\limsup_{t\to \infty} e^t(8\pi - U(t)) \le 8\pi \frac{\m}{\cc(\Omega)}.
\end{equation}
Indeed, since $U(t) \to 8\pi$ as $t\to \infty$, we can integrate the differential inequality for $U$ to get 
\begin{equation}
\label{ineq1}
8\pi - U(t) \le  \int_t^\infty \frac{(16\pi)^{3/2} m_H(\Sigma_t)}{4\area_g(\Sigma_t)^{1/2}} \, dt,
\end{equation}
where $m_H(\Sigma_t)$ denotes the Hawking mass of $\Sigma_t$ in the $g$ metric. 
As $t\to \infty$, we expect that $m_H(\Sigma_t) \to \m$ and that $\area_g(\Sigma_t)$ is very nearly $4\pi \cc(\Omega)^2 e^{2t}$. Substituting these relations into (\ref{ineq1}) yields inequality (\ref{asy1}). In Section \ref{asymptotic}, we will rigorously justify inequality (\ref{asy1}).

In Section \ref{results}, we prove the main results. 
Theorem \ref{main} follows easily by combining (\ref{eq-U}) and (\ref{asy1}). 
The volumetric Penrose inequality is an immediate corollary since spherical symmetrization decreases the conformal capacity.  
It then remains to prove Theorem \ref{stability} on the stability of the volumetric Penrose inequality.  For simplicity, we describe the proof assuming that $u$ has no critical points.  In this case, we show that the excess in the volumetric Penrose inequality is bounded below by a quantity of the form 
\[
\int_0^\infty w(\tau) \left[\int_{\Sigma_\tau} \vert \mathring A\vert^2\, da\right]\, d\tau,
\]
where $w$ is a smooth function with $w(0) > 0$. 
Hence, if equality almost holds in the volumetric Penrose inequality, then there is a small $s > 0$ for which
\[
\int_{\Sigma_s} \vert \mathring A\vert^2 \, da
\]
is nearly zero. By the work of De Lellis and M\"uller \cite{de2005optimal}, this implies that $\Sigma_s$ is nearly round. Therefore the region $\Omega_s$ enclosed by $\Sigma_s$ has small Fraenkel asymmetry. Finally, the conclusion follows by showing that $\vol(\Omega \operatorname{\Delta} \Omega_s)$ is small, so that $\Omega$ also has small Fraenkel asymmetry.

\subsection{Acknowledgements} The authors would like to thank Xin Zhou for his continual support and guidance. L.M. is supported by an AMS Simons Travel Grant. X.Y. is supported by Hutchinson Fellowship.

\section{Preliminaries} 
In this section, we define the conformal capacity of a bounded set $\Omega \subset \R^3$.  We then define the class of conformally flat, asymptotically flat manifolds relevant in our main theorems. Next, we introduce the function $U$ and study some of its basic properties. Finally, we study the model case where $M$ is Schwarzschild.

\subsection{Conformal Capacity} 
We follow \cite{colesanti2005brunn} to define the conformal capacity for subsets of $\R^3$. Let $\Omega$ be a smooth, bounded domain in $\R^3$. Let $M = \R^3\setminus \Omega$. There exists a weak unique solution $u$ to the PDE 
\[
\begin{cases}
\lap_3 u = 0, &\text{in } M\\
u = 0, &\text{on } \Sigma \\
\frac{u(x)}{\log \vert x\vert} \to 1, &\text{as } \vert x\vert \to \infty. 
\end{cases}
\]
For a proof of this fact, see Section 2 and Remark 2.4 in  \cite{colesanti2005brunn}.
We will call $u$ the conformal capacity function associated to $\Omega$. It follows from work of Kichenassamy and V\'eron \cite[Theorem 1.1 and Remarks 1.4-1.5]{kichenassamy1986singular} that the conformal capacity function has an asymptotic expansion  
\begin{equation}
\label{expansion}
\begin{cases}
u(x) = \log \vert x\vert + a + o(1),\\
\grad u(x) = \grad \log \vert x\vert + o(\vert x\vert^{-1})
%\del^2 u(x) = \del^2 \log \vert x\vert + o(\vert x\vert^{-2}),
\end{cases}
\end{equation}
for some constant $a\in \R$. Here all quantities are computed with respect to the Euclidean metric. 

\begin{defn}
\label{conformal-capacity}
The conformal capacity of $\Omega$ is $\cc(\Omega) = e^{-a}$. 
\end{defn}

\begin{ex}
Let $\Omega = B_R$ be a ball of radius $R$ centered at the origin. It is easy to check that 
\[
u = \log\left(\frac{r}{R}\right) = \log r - \log R
\]
is the conformal capacity function associated to $B_R$. Therefore we have $\cc(B_R) = e^{-(-\log R)} = R$. 
\end{ex}

Next we compute some quantities associated to $u$ in the asymptotic regime. All calculations are with respect to the Euclidean metric.  

\begin{defn}
For any regular value $t$ of $u$, let $\Sigma_t = \{u=t\}$. 
\end{defn}

The asymptotic expansion for $u$ ensures that every sufficiently large $t$ is a regular value of $u$. Moreover, $\Sigma_t$ is a smooth $2$-sphere when $t$ is sufficiently large. 

\begin{prop}
\label{mc}
The Euclidean mean curvature of the level sets $\Sigma_t$ satisfies 
\[
H = \frac{2}{\cc(\Omega) e^t}(1 + o(1)).
\]
\end{prop}

\begin{proof}
Since $u$ is 3-harmonic, one has 
\[
0 = \lap_3 u = \div(\vert \grad u\vert \grad u) = \grad \vert \grad u\vert \cdot \grad u + \vert \grad u\vert \lap u. 
\]
Hence the mean curvature of $\Sigma_t$ satisfies 
\[
H = \div\left(\frac{\grad u}{\vert \grad u\vert}\right) = \frac{\lap u}{\vert \grad u\vert} - \frac{\grad \vert \grad u\vert \cdot \grad u}{\vert \grad u\vert^2} = -2\frac{\grad \vert \grad u\vert \cdot \grad u}{\vert \grad u\vert^2}. 
\]
The result then follows from the asymptotic expansion (\ref{expansion}), noting that $\vert x\vert = \cc(\Omega)e^{t}(1+o(1))$.
\end{proof}

The previous proposition implies that $\Sigma_t$ is outer-minimizing in the Euclidean metric for sufficiently large $t$. Indeed, the level sets of $u$ form a mean-convex foliation of the region lying outside of $\Sigma_t$. 

\begin{prop}
The Euclidean area of the level sets of $u$ satisfies 
\[
 \area(\Sigma_t) = 4\pi \cc(\Omega)^2 e^{2t}(1+o(1))
\]
\end{prop}

\begin{proof}
Recall that $\cc(\Omega) = e^{-a}$. Therefore, the asymptotic expansion (\ref{expansion}) implies that $t$ and $\vert x\vert$ are related by 
$
\vert x\vert = \cc(\Omega) e^{t}(1+o(1)). 
$
In particular, we have 
\[
C(t)\inv \cc(\Omega) e^t \le \vert x\vert \le C(t) \cc(\Omega) e^t
\]
for a constant $C(t) > 1$ which converges to 1 as $t \to \infty$. Hence the level set $\Sigma_t$ lies between the concentric spheres of radius $C(t)\inv \cc(\Omega) e^t$ and $C(t) \cc(\Omega) e^t$ provided $t$ is large enough. Since the coordinate spheres are outer-minimizing and $\Sigma_t$ is also outer-minimizing, this implies that 
\[
4\pi C(t)^{-2} \cc(\Omega)^2 e^{2t} \le \area(\Sigma_t) \le 4 \pi C(t)^2 \cc(\Omega)^2 e^{2t} 
\]
It follows that $\area(\Sigma_t) = 4\pi \cc(\Omega)^2 e^{2t}(1+o(1))$, as needed. 
\end{proof}

\begin{prop}
We have $\int_{\Sigma_t} \vert \grad u\vert^2\, da = 4\pi$ for all regular values $t$. 
\end{prop}

\begin{proof}
Since $u$ is $3$-harmonic, this quantity is a constant independent of $t$. On the other hand, the asymptotic expansions for $\grad u$ and $\area(\Sigma_t)$ imply that 
\[
\int_{\Sigma_t} \vert \grad u\vert^2\,da = \left(4\pi \cc(\Omega)^2 e^{2t} \right)\left(\frac{1}{\cc(\Omega)^2 e^{2t}} \right)(1+o(1)). 
\]
The result follows by taking the limit as $t\to \infty$. 
\end{proof}

\subsection{Conformally Flat, Asymptotically Flat Manifolds}

We want to investigate the relationship between the conformal capacity and the mass of an asymptotically flat, conformally flat manifold. 

\begin{defn}
Let $\mathcal M$ be the class of all Riemannian manifolds $(M,g)$ satisfying the following conditions:
\begin{itemize}
\item[(i)] $M = \R^3\setminus \Omega$ for some smooth domain $\Omega\subset \R^3$ diffeomorpic to a ball;
\item[(ii)] $g = f^4 g_{\text{euc}}$ for some smooth $f > 0$;
\item[(iii)] there is a constant $C > 0$ such that $\vert f - 1\vert \le C\vert x\vert^{-1}$, $\vert f_i\vert \le C \vert x\vert^{-2}$, and $\vert f_{ij}\vert \le C \vert x\vert^{-3}$ for sufficiently large $\vert x\vert$;
\item[(iv)] $M$ has non-negative scalar curvature;
\item[(v)] $\Sigma = \bd M$ is a minimal 2-sphere with respect to $g$.
\end{itemize}
The estimates in (iii) ensure that $(M,g)$ is asymptotically flat in the usual sense. 
The requirement that $M$ has non-negative scalar curvature is equivalent to $\lap_{\text{euc}} f \le 0$. 
\end{defn}

\begin{defn}
The ADM mass of an asymptotically flat manifold $M$ is the quantity 
\[
\m = \lim_{r\to \infty} \frac{1}{16\pi} \int_{S_r} \sum_{i,j}(g_{ij,i}-g_{ii,j}) \nu^j\, da.
\]
Here $S_r$ denotes the sphere of radius $r$ centered at the origin, $\nu$ is the unit outward normal to $S_r$, and all quantities are computed with respect to the Euclidean background metric. 
\end{defn}

% \begin{prop} For $(M,g)\in \mathcal M$, the ADM mass is equivalently given by 
% \[
% \m = -\lim_{r\to \infty} \frac{1}{2\pi} \int_{S_r} f_r\, da,
% \]
% where $f_r$ denotes the Euclidean radial derivative of $f$.  
% \end{prop}

% \begin{proof}
% This is a straightforward calculation. 
% \end{proof}

\subsection{The $U$ Function} Fix some $(M,g) \in \mathcal M$. Let $u$ be the conformal capacity function for $\Omega\subset \R^3$.

\begin{defn}
For each regular value $t$ of $u$, define 
\[
U(t) = \int_{\Sigma_t} H_g \vert \grad^g u\vert \, da_g,
\]
where all geometric quantities are computed with respect to the $g$-metric. 
\end{defn} 

The $U$ functional can equivalently be expressed in terms of the Euclidean metric via 
\[
U(t) = \int_{\Sigma_t} H\vert \grad u\vert + \frac{4 f_\nu}{f} \vert \grad u\vert\, da,
\]
where $\nu$ is the outer unit normal to $\Sigma_t$. This is straightforward to verify using the transformation rules for various geometric quantities under a conformal change. 

\begin{prop}
\label{8pi}
We have $U(t) \to 8\pi$ as $t\to \infty$. 
\end{prop}

\begin{proof}
We will use the formula for $U$ in terms of the Euclidean metric. The asymptotics for $f$ and $\vert \grad u\vert$ and $\area(\Sigma_t)$ imply that 
\[
\int_{\Sigma_t} \frac{f_\nu}{f}\vert \grad u\vert\, da \to 0
\]
as $t\to \infty$. On the other hand, we have 
\[
\int_{\Sigma_t} H\vert \grad u\vert \, da = \left(4\pi \cc(\Omega)^2 e^{2t} \right) \left(\frac{2}{\cc(\Omega) e^t} \right) \left(\frac{1}{\cc(\Omega)e^t} \right) (1+o(1)).
\]
Therefore one has 
\[
\int_{\Sigma_t} H\vert \grad u\vert \, da \to 8\pi
\]
as $t\to\infty$ and the result follows. 
\end{proof}

\subsection{The Model Case} 

We conclude this section by studying the model case of Schwarzschild. 
Recall that the mass $m$ Schwarzschild manifold is given by 
\[
(M,g) = \left(\R^3\setminus B_{\frac m 2}, \left(1+\frac{m}{2r}\right)^4 g_{\text{euc}}\right).
\]
We have 
\[
u = \log\left(\frac{2r}{m}\right) = \log r - \log\left(\frac{m}{2}\right) 
\]
and $\cc(B_{\frac m 2}) = \frac{m}{2}$. Thus Schwarzschild is an equality case for Theorem \ref{main}. 

On Schwarzschild, we have $f = 1 + \frac{m}{2r}$. Moreover, $r$ is related to $t$ via $r = \frac{m}{2}e^t$. Therefore, using the Euclidean formula for $U$, it is straightforward to compute that 
\begin{align*}
U(t) &= \int_{\bd B_{\frac{m}{2}e^t}}  H\vert \grad u\vert + \frac{4 f_r}{f}\vert \grad u\vert \, da \\
&= (\pi m^2 e^{2t})\left(\frac{8}{m^2 e^{2t}} - \frac{16}{m^2 e^{2t}(e^t + 1)} \right) = 8\pi \left(\frac{e^t - 1}{e^t+1}\right). 
\end{align*}
In particular, this function does not depend on $m$. 

\section{The Smooth Setting} 
\label{smooth}

In this section, we prove a monotonicity formula holding along the level sets of a $3$-harmonic function $u$ in the absence of critical points. 
Fix some $(M,g)\in \mathcal M$. 
Let $u$ be the conformal capacity function for $\Omega$.  Since the 3-Laplacian is conformally invariant, $u$ satisfies the equation $\lap_3 u = 0$ with respect to both the Euclidean metric and the $g$ metric. Throughout this section,  we assume that $u$ has no critical points. In this case, $\Sigma_t$ is a smooth 2-sphere for all $t\ge 0$. 

\begin{prop}
\label{smooth-monotone}
The function $U$ satisfies the differential inequality 
\[
U' + \frac{U^2}{16\pi} \le 4\pi. 
\]
\end{prop}

\begin{proof} All of the following calculations are understood to take place with respect to the $g$-metric.  We have chosen not to explicitly indicate this to avoid cluttering the notation. 
As in the proof of Proposition \ref{mc}, the fact that $u$ is $3$-harmonic implies that 
\[
H = -2\frac{\grad\vert \grad u\vert \cdot \grad u}{\vert \grad u\vert^2}. 
\]
We have 
\begin{align*}
U'(t) &= \int_{\Sigma_t} H \frac{\grad \vert \grad u\vert \cdot \grad u}{\vert \grad u\vert^2} + H^2 - \vert \grad u\vert \left(\lap_{\Sigma_t} \frac{1}{\vert \grad u\vert} + (\vert A\vert^2 + \ric(\nu,\nu))\frac{1}{\vert \grad u\vert}\right)\, da\\
&= \int_{\Sigma_t} \frac{H^2}{2} - \frac{\vert \grad^{\Sigma_t}\vert \grad u\vert \vert^2}{\vert \grad u\vert^2} - \vert A\vert^2 - \ric(\nu,\nu)\, da\\
&= \int_{\Sigma_t} K - \frac{\vert \grad^{\Sigma_t}\vert \grad u\vert \vert^2}{\vert \grad u\vert^2} - \frac{H^2}{4} - \frac { \vert \mathring A\vert^2}{2} - \frac{R}{2} \, da\\
&\le 4\pi - \frac{1}{4}\int_{\Sigma_t} H^2\, da,
\end{align*} 
where we used the Gauss-Bonnet theorem to get to the last line. 
Next observe that 
\[
U(t)^2 = \left(\int_{\Sigma_t} H\vert \grad u\vert\, da\right)^2 \le \left(\int_{\Sigma_t} \vert \grad u\vert^2\, da\right) \left(\int_{\Sigma_t} H^2\, da\right) = 4\pi \int_{\Sigma_t} H^2\, da. 
\]
Therefore, we obtain 
\[
U'(t) \le 4\pi - \frac{U(t)^2}{16\pi},
\]
as needed. 
\end{proof}

\begin{prop}
The function $U$ satisfies 
\[
U(t) \le 8\pi \left(\frac{e^t - 1}{e^t + 1}\right). 
\] 
\end{prop}

\begin{proof}
First observe that $U(0) = 0$ since $\Sigma$ is minimal in the $g$ metric. Next recall that $U(t) \to 8\pi$ as $t\to \infty$ by Proposition \ref{8pi}. The function 
\[
U_s(t) = 8\pi \left(\frac{e^t - 1}{e^t + 1}\right)
\]
solves 
\[
U_s' + \frac{U_s^2}{16\pi} = 4\pi
\]
with $U_s(0) = 0$ and $U_s(t) \to 8\pi$ as $t\to \infty$. As mentioned above, $U_s$ is nothing but the $U$ function in the model case of Schwarzschild. 

Suppose for contradiction that the conclusion of the proposition fails. Then there is a time $t$ where $U - U_s$ attains a positive maximum. At this time $t$ we have 
\[
4\pi =  U_s' + \frac{U_s^2}{16\pi} <  U' + \frac{U^2}{16\pi} \le 4\pi,
\] 
which is a contradiction. 
\end{proof}

\section{Monotonicity in the Presence of Critical Points} 
\label{critical}
In this section, we aim to show that monotonicity still holds in the presence of critical points.
We closely follow the strategy developed by Agostiniani-Mantegazza-Mazzieri-Oronzio \cite{agostiniani2022riemannian}.  

Supposing that $t$ is a regular value of $u$, define the following quantity:
\begin{align*}
    Q(t):=8\pi(e^t+2-e^{-t})-(e^t+1)^2e^{-t}U(t).
\end{align*}
In this section, we prove the following monotonicity formula.
\begin{prop}\label{prop: main monotonicity}
    Suppose $(M^3,g)$ is an asymptotically flat manifold with compact minimal boundary $\Sigma$. Assume further that $R_g\geq 0$, $H_2(M^3,\Sigma)=0$. Let $u$ be the $3$-harmonic function defined on $(M^3,g)$ which is $0$ on $\Sigma$ and approaches $\log|x|$ as $|x|\to \infty$. Then for any regular values $0\leq s<t\leq\infty$, we have that
    \begin{align}
        Q(t)\geq Q(s).
    \end{align}
    The equality holds if and only if $(M^3,g)$ is Schwarzschild outside $\Sigma_s$.
\end{prop}

Note that $t = 0$ is always a regular value of $u$ by the Hopf-type lemma for $p$-harmonic functions \cite{tolksdorf1983dirichletproblem}. Hence as a direct corollary, by comparing $Q(t)$ with $Q(0)$, we obtain the following.
\begin{corollary}
\label{U-bound}
    Suppose $(M^3,g)$ is an asymptotically flat manifold with compact minimal boundary $\Sigma$. Assume further that $R_g\geq 0$, $H_2(M^3,\Sigma)=0$. Let $u$ be the $3$-harmonic function defined on $(M^3,g)$ which is $0$ on $\Sigma$ and approaches $\log|x|$ as $|x|\to \infty$. Then for any regular values $t\in(0,\infty]$, we have that
    \begin{align}
        U(t)\leq 8\pi\left(\frac{e^t-1}{e^t+1}\right).
    \end{align}
\end{corollary}

\subsection{No critical points}

First consider the simplest case when $u$ has no critical points. Define the vector field
\begin{align*}
    X:=2(e^u+2-e^{-u})|\nabla u|\nabla u-(e^u+1)^2e^{-u}\left(\frac{\Delta u}{|\nabla u|}\nabla u-\nabla|\nabla u|\right).
\end{align*}
By the definition of $Q(t)$, we can rewrite $Q(t)$ as
\begin{align*}
    Q(t)=\int_{\Sigma_t}\langle X,\frac{\nabla u}{|\nabla u|}\rangle.
\end{align*}
Then by Bochner's formula, we compute that
\begin{align*}
    \div X=&2(e^u+e^{-u})|\nabla u|^3+2(e^u-e^{-u})\langle \nabla |\nabla u|,\nabla u\rangle\\
    &+(e^u+1)^2e^{-u}\left(\frac{|\nabla^{\Sigma_t}|\nabla u||^2}{|\nabla u|^2}+\|A\|^2+\ric(\nu,\nu)\right)|\nabla u|-2(e^u+1)^2e^{-u}\frac{\langle \nabla |\nabla u|,\nabla u\rangle^2 }{|\nabla u|^3}.
\end{align*}
Now, with the divergence theorem and the traced Gauss equation, we have that
\begin{align*}
    Q(t)-Q(s)&=\int_{\Sigma_t}\langle X,\frac{\nabla u}{|\nabla u|}\rangle da_g-\int_{\Sigma_s}\langle X,\frac{\nabla u}{|\nabla u|}\rangle da_g\\
             &=\int_{\{s\leq u\leq t\}}\div XdV_g\\
             &=\int_s^t\int_{\Sigma_{\tau}}\frac{\div X}{|\nabla u|}da_g\\
             &\geq \int_s^t\left[\int_{\Sigma_{\tau}} 2(e^{\tau}+e^{-\tau})|\nabla u|^2-(e^{\tau}+1)^2e^{-\tau}\int_{\Sigma_{\tau}}K_{\Sigma_{\tau}}+2(e^{\tau}-e^{\-\tau})\int_{\Sigma_{\tau}}|\nabla^{\perp}|\nabla u||\right]d\tau\\
             &\qquad \quad +\int_s^{t}(e^{\tau}+1)e^{-\tau}\left(\int_{\Sigma_{\tau}}\frac{|\nabla^{\perp}|\nabla u||^2}{|\nabla u|^2}\right)d\tau+\frac{1}{2}\int_s^t(e^{\tau}+1)^2e^{-\tau}\left(\int_{\Sigma_r}\vert \mathring{A} \vert^2\right)d\tau\\
             &\geq \int_s^{t}\int_{\Sigma_{\tau}}\left[(e^{\tau}-2+e^{-\tau})|\nabla u|^2+2(e^{\tau}-e^{-\tau})|\nabla^{\perp}|\nabla u||+(e^{\tau}+2+e^{-\tau})\frac{|\nabla^{\perp}|\nabla u||^2}{|\nabla u|^2}\right] d\tau\\
             &\qquad \quad +\frac{1}{2}\int_s^t(e^{\tau}+1)^2e^{-\tau}\left(\int_{\Sigma_r}\vert \mathring{A} \vert^2\right)d\tau\\             &=\int_s^{t}e^{-\tau}\left\{\int_{\Sigma_{\tau}}\left[(e^{\tau}-1)|\nabla u|-(e^{\tau}+1)\frac{|\nabla^{\perp}|\nabla u||}{|\nabla u|}\right]^2\right\}d\tau+\frac{1}{2}\int_s^t(e^{\tau}+1)^2e^{-\tau}\left(\int_{\Sigma_r}\vert \mathring{A} \vert^2\right)d\tau\\
             &\geq 0.
\end{align*}
This proves the monotonicity of $Q$. 

\subsection{Monotonicity with negligible critical points}
Next we aim to prove the monotonocity of $Q$ function, under the assumption that the set of critical values of $u$ is measure zero.  We closely follow the strategy in \cite{agostiniani2022riemannian}. For the application to the stability of the volumetric Penrose inequality, we need to keep track of the second fundamental form term appearing in the monotonicity formula.  We present the details for the reader's convenience.

Consider the vector field $X$ defined earlier and rewrite it as
\[
X=2\left(e^u+2-e^{-u}\right)|\nabla u|\nabla u+Y.
\]
We note that $Y$ is only well defined away from the set of critical points. Consider a sequence of cut-off functions $\eta_{k}:[0,\infty)\to [0,1]$ satisfying:
\[
\eta_{k}(\tau)\equiv 0\quad \forall \tau\in [0,\frac{1}{2k}], \quad 0\leq \eta_k'(\tau)\leq 2k\quad \forall \tau\in[\frac{1}{2k},\frac{3}{2k}],\quad \eta_{k}(\tau)\equiv 1\quad \forall \tau\in[\frac{3}{2k},\infty). 
\]
With the cut-off functions, we define the  vector fields
\[
Y_k:=\eta_k(|\nabla u|)Y,\quad X_k=2\left(e^{2}+2-e^{-u}\right)|\nabla u|\nabla u+Y_k.
\]
Note that 
\begin{align*}
\eta_k'(|\nabla u|)\langle \nabla|\nabla u|,Y\rangle&=-(e^u+1)^2e^{-u}\eta_k'(|\nabla u|)\left(\frac{\Delta u}{|\nabla u|}\langle\nabla |\nabla u|,\nabla u\rangle-|\nabla|\nabla u||^2\right),\\
&=(e^u+1)^2e^{-u}\eta_k'(|\nabla u|)\left(\langle \nabla|\nabla u|,\frac{\nabla u}{|\nabla u|}\rangle^2+|\nabla|\nabla u||^2\right)\\
&\geq 0,
\end{align*}
provided $\eta_k'(|\nabla u|)\geq 0$, which follows from the definition.

Thus we obtain
\[
\div X_k\geq 2\left(e^u+e^{-u}\right)|\nabla u|^3+\eta_{k}(|\nabla u|)\div Y.
\]
Now for regular values $0<s<t<\infty$, we can find $k_0\in\mathbb Z_+$ such that for any 
$k\geq k_0$, we have
\begin{align*}
    Q(t)-Q(s)&=\int_{\Sigma_t}\langle X_k,\frac{\nabla u}{|\nabla u|}\rangle da_g-\int_{\Sigma_s}\langle X_k,\frac{\nabla u}{|\nabla u|}\rangle da_g\\
    &=\int_{\{s\leq u\leq t\}}\div X_k dV_g\\
    &\geq \int_{\{s\leq u\leq t\}}\left[2(e^u+e^{-u})|\nabla u|^3+\eta_k(|\nabla u|)\div Y\right]dV_g.
\end{align*}

Away from critical points of $u$, we can rewrite $\div Y$ as $P+D$: 
\begin{align*}
P&=(e^u+1)^2e^{-u}\left(\frac{|\nabla^{\Sigma_t}|\nabla u||^2}{|\nabla u|^2}+\vert\mathring{A}\vert^2\right),\\
D&=2(e^u-e^{-u})\langle\nabla|\nabla u|,\nabla u\rangle+(e^u+1)^2e^{-u}\ric(\frac{\nabla u}{|\nabla u|},\frac{\nabla u}{|\nabla u|}),
\end{align*}
where we used the Bochner's formula.
Note that $P$ is non-negative, and $D$ is bounded on every compact subset of $M$, so by applying the dominating convergence theorem to $\eta_k D$ and monotone convergence theorem to $\eta_k P$, we can  pass $k\to\infty$ and obtain that
\begin{align*}
    Q(t)-Q(s)&\geq \int_{\{s\leq u\leq t\}}\chi_{M\setminus \text{Crit}(u)}\div XdV_g.
\end{align*}
Similar arguments were used in \cite{agostiniani2022riemannian}.

Let $\mathcal N$ denote the set of critical values of $u$. Now we can apply the co-area formula and the traced Gauss equation to get
\begin{align*}
    Q(t)-Q(s)&\geq \int_{[s,t]\setminus \mathcal {N}}\left(\int_{\Sigma_{\tau}}\frac{\div X}{|\nabla u|}\right)d\tau\\
    &\geq \int_{[s,t\setminus\mathcal{N}]}\left[\int_{\Sigma_{\tau}} 2(e^{\tau}+e^{-\tau})|\nabla u|^2-(e^{\tau}+1)^2e^{-\tau}\int_{\Sigma_{\tau}}K_{\Sigma_{\tau}}+2(e^{\tau}-e^{\-\tau})\int_{\Sigma_{\tau}}|\nabla^{\perp}|\nabla u||\right]d\tau\\
             &\qquad \quad+\int_{[s,t]\setminus\mathcal{N}}(e^{\tau}+1)e^{-\tau}\left(\int_{\Sigma_{\tau}}\frac{|\nabla^{\perp}|\nabla u||^2}{|\nabla u|^2}\right)d\tau+\frac{1}{2}\int_s^t(e^{\tau}+1)^2e^{-\tau}\left(\int_{\Sigma_r}\vert \mathring{A} \vert^2\right)d\tau\\
             &\geq \int_{[s,t]\setminus\mathcal{N}}\int_{\Sigma_{\tau}}\left[(e^{\tau}-2+e^{-\tau})|\nabla u|^2+2(e^{\tau}-e^{-\tau})|\nabla^{\perp}|\nabla u||+(e^{\tau}+2+e^{-\tau})\frac{|\nabla^{\perp}|\nabla u||^2}{|\nabla u|^2}\right] d\tau\\
             &\qquad\quad +\frac{1}{2}\int_{[s,t]\setminus\mathcal{N}}(e^{\tau}+1)^2e^{-\tau}\left(\int_{\Sigma_r}\vert \mathring{A} \vert^2\right)d\tau\\             
             &=\int_{[s,t]\setminus\mathcal{N}}e^{-\tau}\left\{\int_{\Sigma_{\tau}}\left[(e^{\tau}-1)|\nabla u|-(e^{\tau}+1)\frac{|\nabla^{\perp}|\nabla u||}{|\nabla u|}\right]^2\right\}d\tau\\
             &\qquad\quad +\frac{1}{2}\int_{[s,t]\setminus\mathcal{N}}(e^{\tau}+1)^2e^{-\tau}\left(\int_{\Sigma_r}\vert \mathring{A} \vert^2\right)d\tau\\
             &\geq 0.   
\end{align*}
From the second line to the third line, we used the assumption that $H_2(M,\Sigma)=0$, which implies that all regular level sets of $u$ are connected. 

\subsection{An approximate monotonicity formula}\label{sec:epsilon-approx}
Since $3$-harmonic functions are only $C^{1,\alpha}$ in general, we cannot apply the Sard's theorem to exclude the case when the set of critical points is not measure zero.
In this subsection, we consider the most general case, with no assumption the set of critical points of $u$. As in \cite{agostiniani2022riemannian}, we proceed by $\epsilon$-approximation, which was first introduced in \cite{dibenedetto1983c1+}. Again, for later application to the stability problem, it is important to keep track of the 2nd fundamental form term appearing in the monotonicity formula. 

Suppose $T$ is a fixed regular value of $u$, and let $u_{\epsilon}$ denote the solution to the $\epsilon$-perturbed version of the $3$-harmonic equation:
\[
\begin{cases}
\div\left(\sqrt{|\nabla u_{\epsilon}|^2+\epsilon^2}\nabla u_{\epsilon}\right)  = 0, &\text{in } M_T:=\{0\leq u\leq T\}\\
u_{\epsilon} = 0, &\text{on } \Sigma \\
u_{\epsilon}=T, &\text{on } \{u=T\}. 
\end{cases}
\]
For each regular value $\tau$ of $u_{\epsilon}$, we use $\Sigma^{\epsilon}_{\tau}$ to denote the level set $\{u_{\epsilon} = \tau\}$.

By the divergence theorem, we have
\[
\int_{\Sigma^{\epsilon}_{t}}\sqrt{|\nabla u_{\epsilon}|^2+\epsilon^2}|\nabla u_{\epsilon}|da_g=4\pi(1+f(\epsilon)),
\]
where $f(\epsilon)$ satisfies $\lim_{\epsilon\to 0} f(\epsilon)=0$. The convergence of $f(\epsilon)$ to 0 follows from the fact that 
$u_{\epsilon}$ converges uniformly to $u$ in $W^{1,3}(M_T)$; see \cite{dibenedetto1983c1+}.

We then consider the vector field 
\[
X_{\epsilon}=2(e^{u_{\epsilon}}+2-e^{-{u_{\epsilon}}})\sqrt{|\nabla u_{\epsilon}|^2+\epsilon^2}\nabla u_{\epsilon}+Y_{\epsilon},
\]
where 
\[
Y_{\epsilon}=-(e^{u_{\epsilon}}+1)^2e^{-u_{\epsilon}}\left(\frac{\Delta u_{\epsilon}}{|\nabla u_{\epsilon}|}\nabla u_{\epsilon}-\nabla|\nabla u_{\epsilon}|\right).
\]
Suppose $x\in M_{T}$ is a regular point of $u_{\epsilon}$. Using the same techniques as before, we compute that
\[
\div X_{\epsilon}=2(e^{u_{\epsilon}}+e^{-u_{\epsilon}})\sqrt{|\nabla u_{\epsilon}|^2+\epsilon^2}|\nabla u_{\epsilon}|^2+\div Y_{\epsilon},
\]
and by Bochner's formula, we have
\begin{align*}
    \div Y_{\epsilon}&=2(e^{u_{\epsilon}}-e^{-u_{\epsilon}})\frac{(|\nabla u_{\epsilon}|^2+\epsilon^2/2)}{|\nabla u_{\epsilon}|^2+\epsilon^2}\langle\nabla |\nabla u_{\epsilon}|,\nabla u_{\epsilon}\rangle\\
    &+(e^{u_{\epsilon}}+1)^2e^{-u_{\epsilon}}|\nabla u_{\epsilon}|\left(\|A\|^2+\ric(\frac{\nabla u_{\epsilon}}{|\nabla u_{\eps}|},\frac{\nabla u_{\eps}}{|\nabla u_{\eps}|})+\frac{|\nabla^{\Sigma_t}|\nabla u_{\epsilon}||^2}{|\nabla u_{\epsilon}|^2}\right)\\
    &-(e^{u_{\eps}}+1)^2e^{-u_{\eps}}\frac{|\nabla u_{\eps}|(2|\nabla u_{\eps}|^2+\eps^2)}{(|\nabla u_{\eps}|^2+\eps^2)^2}|\nabla^{\perp}|\nabla u_{\eps}||^2.
\end{align*}
Using the same cut-off functions as in the previous subsection, we define
\begin{align*}
    X_{k,\epsilon}=2(e^{u_{\epsilon}}+2-e^{-u_{\epsilon}})\sqrt{|\nabla u_{\epsilon}|^2+\epsilon^2}\nabla u_{\epsilon}+Y_{k,\epsilon},
\end{align*}
where
\[
  Y_{k,\epsilon}=\eta_{k}(|\nabla u_{\epsilon}|)Y_{\epsilon}.
\]
As in the previous subsection, we have that
\[
\eta_{k}'(|\nabla u_{\eps}|)\langle \nabla|\nabla u_{\eps}|,Y_{\eps}\rangle\geq 0,
\]
which implies that
\[
\div Y_{k,\eps}\geq \eta_k(|\nabla u_{\eps}|)\div Y_{\eps}.
\]
Supposing $t$ is a regular value of $u_{\epsilon}$, we define
\[
Q_{\epsilon}(t):=\int_{\Sigma_t^{\epsilon}}\langle X_\epsilon,\frac{\nabla u_{\epsilon}}{|\nabla u_{\epsilon}|}\rangle da_g.
\]
For any regular values $0<s<t<\infty$, there exists $k_0\in\mathbb Z_+$, such that for any $k\geq k_0$, we have that
\begin{align*}
    &Q_{\epsilon}(t)-Q_{\epsilon}(s)\\
    &=\int_{\Sigma^{\epsilon}_t}\langle X_{k,\epsilon},\frac{\nabla u_{\epsilon}}{|\nabla u_{\epsilon}|}\rangle da_g-\int_{\Sigma^{\epsilon}_s}\langle X_{k,\epsilon},\frac{\nabla u_{\epsilon}}{|\nabla u_{\epsilon}|}\rangle da_g\\
    &=\int_{\{s\leq u_{\epsilon}\leq t\}}\div X_{k,\epsilon}\\
    &\geq \int_{\{s\leq u_{\eps}\leq t\}}2(e^{u_{\eps}}+e^{-u_{\eps}})\eta_k(|\nabla u_{\eps}|)\sqrt{|\nabla u_{\eps}|^2+\eps^2}|\nabla u_{\eps}|^2+\eta_k(|\nabla u_{\eps}|)\div Y_{\eps}.
\end{align*}
We rewrite $\div Y_{\eps}$ as $P_{\eps}+D_{\eps}$ for any point outside the critical point set, where
\begin{align*}
    P_{\eps}&=(e^{u_{\eps}}+1)^2e^{-u_{\eps}}\left(\vert\mathring{A}\vert^2+\frac{|\nabla^{\Sigma_t}|\nabla u_{\eps}||^2}{|\nabla u_{\eps}|^2}+\frac{\eps^2(|\nabla u_{\eps}|^2+\eps^2/2)}{(|\nabla u_{\eps}|^2)|\nabla u_{\eps}|^2}|\nabla^{\perp}|\nabla u_{\eps}||^2\right),\\
    D_{\eps}&=2(e^{u_{\eps}}-e^{-u_{\eps}})\frac{(|\nabla u_{\eps}|^2+\eps^2/2)}{|\nabla u_{\eps}|^2+\eps^2}\langle \nabla |\nabla u_{\eps}|,\nabla u_{\eps}\rangle+(e^{u_{\eps}}+1)^2e^{-u_{\eps}}\ric(\nu_{\eps},\nu_{\eps}).
\end{align*}
Note that $P_{\eps}$ is non-negative, and $D_{\eps}$ is bounded on every compact subset of $M_T$, so by applying the dominating convergence theorem to $\eta_k D_{\eps}$ and monotone convergence theorem to $\eta_k P_{\eps}$, we can pass $k\to\infty$. Together with the traced Gauss equation and Gauss-Bonnet formula, we obtain that
\begin{align*}
    Q_{\epsilon}(t)-Q_{\epsilon}(s)&\geq 2f(\epsilon)\int_{[s,t]\setminus\mathcal{N}}(e^\tau+e^{-\tau})d\tau\\
    &\qquad \quad +\int_{[s,t]\setminus\mathcal{N}}\int_{\Sigma^{\eps}_{\tau}}(e^{\tau}-1)^2e^{-\tau}\sqrt{|\nabla u_{\eps}|^2+\eps^2}|\nabla u_{\eps}|\\
    &\qquad \quad+\int_{[s,t]\setminus\mathcal{N}}\int_{\Sigma^{\eps}_{\tau}}2(e^{\tau}+1)(e^{\tau}-1)e^{-\tau}\frac{|\nabla u_{\eps}|^2+\eps^2/2}{|\nabla u_{\eps}|^2+\eps^2}\langle\nabla u_{\eps},\frac{\nabla u_{\eps}}{|\nabla u_{\eps}|}\rangle\\
    &\qquad \quad+\int_{[s,t]\setminus\mathcal{N}}\int_{\Sigma^{\eps}_{\tau}} (e^{\tau}+1)^2e^{-\tau}\frac{(|\nabla u_{\eps}|^2+3\eps^2/2)(|\nabla u_{\eps}|^2+\eps^2/2)}{(|\nabla u_{\eps}|^2+\eps^2)^2|\nabla u_{\eps}|^2}|\nabla^{\perp}|\nabla u_{\eps}||^2\\
    &\qquad \quad+\int_{[s,t]\setminus\mathcal{N}}\int_{\Sigma^{\eps}_{\tau}}\frac{1}{2}(e^{\tau}+1)^2e^{-\tau}\vert\mathring{A}\vert^2\\
    &\geq 2f(\epsilon)\int_{[s,t]\setminus\mathcal{N}}(e^\tau+e^{-\tau})d\tau+\int_{[s,t]\setminus\mathcal{N}}\int_{\Sigma^{\eps}_{\tau}}\frac{1}{2}(e^{\tau}+1)^2e^{-\tau}\vert\mathring{A}\vert^2\\
    &\qquad \quad+\int_{[s,t]\setminus\mathcal{N}}\int_{\Sigma^{\eps}_{\tau}}e^{-\tau}\big\{(e^{\tau}+1)\frac{\sqrt{(|\nabla u_{\eps}|^2+3\eps^2/2)(|\nabla u_{\eps}|^2+\eps^2/2)}}{(|\nabla u_{\eps}|^2+\eps^2)|\nabla u_{\eps}|}\nabla^{\perp}|\nabla u_{\eps}|\\
    &\qquad \quad+(e^{\tau}-1)\sqrt{\frac{|\nabla u_{\eps}|^2+\eps^2/2}{|\nabla u_{\eps}|^2+3\eps^2/2}}\nabla u_{\eps}\big\}^2\\   
    &\qquad \quad+\int_{[s,t]\setminus\mathcal{N}}\int_{\Sigma^{\eps}_{\tau}}(e^{\tau}-1)^2e^{-\tau}\frac{\eps^2|\nabla u|^2}{|\nabla u|^2+3\eps^2/2}\\
    &\geq 2f(\epsilon)\int_{[s,t]\setminus\mathcal{N}}(e^\tau+e^{-\tau})d\tau+\int_{[s,t]\setminus\mathcal{N}}\int_{\Sigma^{\eps}_{\tau}}\frac{1}{2}(e^{\tau}+1)^2e^{-\tau}\vert\mathring{A}\vert^2.
    \end{align*}
Thus $Q_\eps$ is almost monotone.

Using the same argument as in \cite{agostiniani2022riemannian}, we obtain the following lemma.
\begin{lem}
    For regular values of $t$,
    \[
    \lim_{\epsilon\to 0}Q_{\epsilon}(t)=Q(t).
    \]
\end{lem}

Combining the above Lemma with the fact that $\lim_{\eps\to 0}|f(\eps)|=0$, the monotonicity formula \textbf{Proposition \ref{prop: main monotonicity}} then follows.

\section{Asymptotics at Infinity} 
\label{asymptotic}

In this section, we show that $U$ detects the mass to conformal capacity ratio as $t\to \infty$. 
First we prove a lemma. 

\begin{lem}
\label{adm}
Fix $(M,g) \in \mathcal M$. Then 
\[
\lim_{s\to \infty} \int_{\Sigma_s} f_\nu\, da = -{2\pi}\m,
\]
where $\nu$ denotes the Euclidean outward unit normal to $\Sigma_s$. 
\end{lem}

\begin{proof}
Bartnik \cite[Proposition 4.1]{bartnik1986mass} proved that the ADM mass can be computed along any sufficiently round exhaustion of $M$. In particular, we have 
\[
\m = \lim_{s\to \infty} \frac{1}{16\pi} \int_{\Sigma_s} \sum_{i,j} (g_{ij,i} - g_{ii,j})\nu^j\, da. 
\]
Now, using the fact that $g = f^4g_{\text{euc}}$, it is easy to compute that 
\[
\sum_{i,j} (g_{ij,i}-g_{ii,j})\nu^j = -8f^3 f_\nu. 
\]
Therefore 
\[
\lim_{s\to \infty} \frac{1}{16\pi} \int_{\Sigma_s} \sum_{i,j} (g_{ij,i}-g_{ii,j})\nu^j\, da = -\lim_{s\to \infty} \frac{1}{2\pi}\int_{\Sigma_s} f_\nu \, da,
\]
and the result follows. 
\end{proof}

The next proposition is the main result of this section. 

\begin{prop}
\label{asy}
Fix $(M,g) \in \mathcal M$. Then 
\[
\limsup_{t\to \infty} e^t(8\pi - U(t)) \le 8\pi \frac{\m}{\cc(\Omega)}. 
\]
\end{prop}

\begin{proof}
Recall that $t$ is a regular value of $u$ whenever $t$ is sufficiently large.  We can re-write $U$ in terms of the Euclidean metric as 
\[
U(t) = \int_{\Sigma_t} H \vert \grad u\vert + \frac{4 f_\nu}{f} \vert \grad u \vert \, da. 
\]
Since $u$ is 3-harmonic with respect to the Euclidean metric, we can calculate as in the proof of Proposition \ref{smooth-monotone} to get 
\[
\frac{d}{dt} \int_{\Sigma_t} H \vert \grad u\vert \, da \le 4\pi\left(1 - \frac{1}{16\pi}\int_{\Sigma_t} H^2 \, da\right) \le 0,
\]
where the last inequality follows from the Euclidean Willmore inequality. 
Therefore we obtain 
\begin{equation}
\label{ineq2} 
8\pi - U(t) = \int_{t}^\infty \frac{d}{ds}U(s)\, ds  \le \int_t^\infty \frac{d}{ds}\left[\int_{\Sigma_s}  \frac{4f_\nu}{f} \vert \grad u\vert \, da\right] \, ds. 
\end{equation}
Observe that 
\[
\nu = \frac{\grad u}{\vert \grad u\vert} = (1+o(1)) \bd_r, \quad 
f_\nu = (1+o(1)) f_r.
\]
It follows that 
\[
\lim_{s\to \infty} \int_{\Sigma_s} \frac{4f_\nu}{f} \vert \grad u\vert\, da = 0.
\]
Hence returning to (\ref{ineq2}), we now have 
\[
8\pi - U(t) \le -\int_{\Sigma_t} \frac{4f_\nu}{f}\vert \grad u\vert\, da = -4 \left(\frac{1}{\cc(\Omega)e^{-t}}\right) (1+o(1))\int_{\Sigma_t} f_\nu\, da.
\]
Therefore, Lemma \ref{adm} shows that 
\[
\limsup_{t\to \infty} e^t(8\pi - U(t)) \le 8\pi \frac{\m}{\cc(\Omega)},
\]
as needed. 
\end{proof}

\section{Proof of Main Theorems}
\label{results}

In this section, we complete the proof of the main results. Theorem \ref{main} and Corollary \ref{vpi} are both straightforward given the formulas that have already been established. We begin with the proof of Theorem \ref{main}.

\begin{theorem}
For any $(M,g) \in \mathcal M$ one has 
$
\m \ge 2 \cc(\Omega).
$ 
Moreover, equality holds if and only if $M$ is Schwarzschild. 
\end{theorem}

\begin{proof} Fix $(M,g)\in \mathcal  M$.  According to Corollary \ref{U-bound}, we have
\[
e^t(8\pi - U(t)) \ge e^t\left(8\pi - 8\pi \left(\frac{e^t - 1}{e^t + 1}\right)\right) = 16\pi \left(\frac{e^t}{e^t+1}\right). 
\]
Passing to the limit as $t\to \infty$ on both sides and using Proposition \ref{asy}, we obtain 
\[
8\pi \frac{\m}{\cc(\Omega)} \ge 16\pi.
\]
This proves that $\m\ge 2\cc(\Omega)$.

It remains to handle the rigidity case. It is easy to verify that equality holds when $M$ is Schwarzschild. On the other hand, suppose that $\m = 2\cc(\Omega)$. By Proposition \ref{prop: main monotonicity}, to show that $M$ is Schwarzschild, it suffices to show that $Q\equiv 0$. Suppose to the contrary that $Q(s) = a > 0$ for some regular value $s$.  Then according to the monotonicity formula for $Q$, we have 
\[
a \le Q(t) = 8\pi(e^t + 2 - e^{-t}) - (e^t+1)^2 e^{-t} U(t) 
\]
for all regular values $t > s$. It follows that 
\[
U(t) \le 8\pi \left(\frac{e^{t}-1}{e^t+1}\right) - \frac{ae^t}{(e^t+1)^2}
\]
for all sufficiently large $t$. Therefore, we have 
\[
e^t(8\pi - U(t)) \ge e^t\left(8\pi - 8\pi\left(\frac{e^{t}-1}{e^t+1}\right) + \frac{ae^t}{(e^t+1)^2}\right) = 16\pi \left(\frac{e^t}{e^t+1}\right) + \frac{ae^{2t}}{(e^t+1)^2}. 
\]
Passing to the limit as $t\to \infty$ and using Proposition \ref{asy} gives 
\[
8\pi \frac{\m}{\cc(\Omega)} \ge 16\pi + a.
\]
This contradicts that $\m = 2\cc(\Omega)$. Therefore $M$ must be Schwarzschild. 
\end{proof}

Now we deduce the volumetric Penrose inequality. 

\begin{corollary}
For any manifold $(M,g)\in \mathcal M$, one has 
\[
\m \ge 2 \left(\frac{3}{4\pi} \vol(\Omega)\right)^{\frac 1 3}. 
\]
Equality holds if and only if $M$ is Schwarzschild. 
\end{corollary}

\begin{proof}
The conformal capacity of $\Omega$ is always at least the conformal capacity of the ball of the same volume; c.f. \cite[Theorem 4.1(i)]{xiao2020geometrical}. Therefore we have 
\[
\m \ge 2 \cc(\Omega) \ge 2\cc(B_R) = 2R 
\]
where $\frac{4\pi}{3}R^3 = \vol(\Omega)$.  It is easy to verify that equality holds when $M$ is Schwarzschild. On the other hand, if equality holds then $\m = 2\cc(\Omega)$, and so $M$ must be Schwarzschild by the rigidity in the mass to conformal capacity inequality.
\end{proof}

Finally we prove the stability theorem for the volumetric Penrose inequality. Recall that $\alpha(\Omega)$ is the Fraenkel asymmetry of $\Omega$.

 \begin{theorem}\label{thm:stability}
There are constants $\eta_0 > 0$ and $C > 0$ such that the following holds.  Consider a manifold $(M,g)\in \mathcal M$. If  
 \[
 \frac{\m}{2\left(\frac{3}{4\pi}\vol(\Omega)\right)^{\frac 1 3}} \le 1+\eta
 \]
 for some $0 < \eta < \eta_0$, then $\alpha(\Omega) \le C\sqrt{\eta}$. 
 \end{theorem}

 \begin{proof}
Suppose $t$ is a regular value of $u$. Then by definition of $Q(t)$ and $U(t)$, we have
    \[
    U(t)=\frac{8\pi e^t (e^t - e^{-t} + 2)}{(e^t+1)^2}-(e^t+1)^{-2}e^tQ(t).
    \]
Now fix a large regular value $T$ for $u$. Then, for any regular value $t \ge T$, we have 
\begin{align*}
e^t(8\pi-U(t)) = \frac{16\pi e^t}{(e^t+1)^2} + \frac{e^{2t}}{(e^t+1)^2}Q(t) \ge  \frac{16\pi e^t}{(e^t+1)^2} + \frac{e^{2t}}{(e^t+1)^2}Q(T).
\end{align*} 
Now, as in Section \ref{sec:epsilon-approx}, let $u_\eps$ be the $\eps$-regularization of $u$ in the domain $\{0\le u \le T\}$. Since $Q_\eps(T) \to Q(T)$ as $\eps\to 0$, we can select $\eps$ small enough that 
\begin{align*}
e^t(8\pi - U(t)) \ge \frac{16\pi e^t}{(e^t+1)^2} + \frac{e^{2t}}{(e^t+1)^2}[Q_\eps(T) - \eta]
\end{align*}
for all regular values $t \ge T$ of $u$.  Likewise, since $Q(0) = 16\pi$, by selecting $\eps$ small enough, we can ensure that $\vert Q_\eps(0)-16\pi\vert \le \eta$. Hence we in fact have 
\[
e^t(8\pi - U(t)) \ge \frac{16\pi e^t}{(e^t+1)^2} + \frac{e^{2t}}{(e^t+1)^2}[Q_\eps(T) - Q_\eps(0) + 16\pi - 2\eta]
\]
for all regular values $t\ge T$ of $u$. Sending $t\to \infty$ in this inequality we obtain 
\[
Q_\eps(T)-Q_\eps(0) \le 8\pi \frac{\m}{\cc(\Omega)} - 16\pi + 2\eta. 
\]
Hence, bounding the capacity from below by the volume and using the assumption of the theorem, we obtain 
\[
Q_\eps(T) - Q_\eps(0) \le C \eta.
\]
Here and in the following $C$ denotes a positive constant (that does not depend on $M$) which is allowed to change from line to line. 

Let $\mathcal N$ denote the set of critical values of $u_\eps$.  Then $\mathcal N$ has measure 0 by Sard's theorem. From Section \ref{sec:epsilon-approx}, we know that 
\begin{align*}
Q_\eps(T) - Q_\eps(0)  \ge 2f(\epsilon)\int_{[s,t]\setminus\mathcal{N}}(e^\tau+e^{-\tau})\, d\tau+\int_{[s,t]\setminus\mathcal{N}}\int_{\Sigma^{\eps}_{\tau}}\frac{1}{2}(e^{\tau}+1)^2e^{-\tau}\vert\mathring{A}\vert^2\, da\, d\tau.
\end{align*}
The first term on the right hand side tends to 0 as $\eps$ tends to 0. Therefore, by selecting $\eps$ small enough, we get 
\[
\int_{[0,T] \setminus \mathcal N} (e^\tau+1)^2 e^{-\tau} \int_{\Sigma^\eps_\tau} \vert \mathring A\vert^2\, da \, d\tau \le C\eta. 
\]
In particular, this implies that 
\[
\int_{[0,\sqrt \eta] \setminus \mathcal N}  \int_{\Sigma^\eps_\tau} \vert \mathring A\vert^2\, da \, d\tau \le C\eta.
\]
Thus there must be a number $s\in [0,\sqrt \eta]\setminus \mathcal N$ for which 
\[
\int_{\Sigma^\eps_s} \vert \mathring A\vert^2\, da \le C\sqrt \eta. 
\]
Let $\Omega^\eps_s$ be the region enclosed by $\Sigma^\eps_s$. By De Lellis and M\"uller \cite{de2005optimal}, there is a point $x\in \R^3$ and a radius $r > 0$ and a conformal paramaterization $\psi\colon S^2\to \Sigma^\eps_s$ satisfying 
 \[
 \|\psi - x - r\text{Id}\|_{W^{2,2}(S^2)} \le C r \|\mathring A\|_{L^2(\Sigma^\eps_s)}^2 \le C r \sqrt{\eta}. 
 \]
 In particular, by the continuous embedding of $W^{2,2}$ into $L^\infty$, this implies that we have the inclusions
 $
 B(x,r(1-C\sqrt{\eta})) \subset \Omega^\eps_s \subset B(x,r(1+C\sqrt \eta)).
 $

The next step is to estimate the value of $r$. Note that $u_\eps \to u$ uniformly as $\eps \to 0$. In particular, by selecting $\eps$ small enough, we can ensure that $\Omega^\eps_s$ is contained in the set $\{u \le 2\sqrt \eta\}$. Since $u$ is the conformal capacity function for $\Omega$, we therefore have 
\[
\cc(\Omega) \le \cc(\Omega^\eps_s) \le \cc(\{u\le 2\sqrt{\eta}\}) = e^{2\sqrt \eta} \cc(\Omega).
\]
By the mass-conformal capacity inequality and the assumption of the theorem, we have
 \[
 2\cc(\Omega) \le \m \le 2(1+\eta)\cc(\Omega).
 \]
 It follows that   
 \[
1-C\sqrt \eta \le e^{-2\sqrt \eta} \le \frac{\m}{2\cc(\Omega_s^\eps)} \le 1 + \eta \le 1+C\sqrt \eta.
 \]
 and therefore that 
 \[
 (1-C\sqrt \eta) \frac{\m}{2} \le \cc(\Omega^\eps_s) \le (1+C\sqrt{\eta})\frac{\m}{2}.
 \]
 Now, since $B(x,r(1+C\sqrt \eta))$ contains $\Omega^\eps_s$, it follows that 
 \[
 r(1+C\sqrt \eta) = \cc(B(x,r(1+C\sqrt \eta))) \ge \cc(\Omega^\eps_s) \ge (1 -C \sqrt{\eta})\frac{\m}{2}.
 \]
 Similarly, since $B(x,r(1-C\sqrt \eta))$ is contained in $\Omega^\eps_s$, it follows that 
 \[
 r(1-C\sqrt\eta)=\cc(B(x,r(1-C\sqrt\eta))) \le \cc(\Omega^\eps_s) \le (1+C\sqrt \eta)\frac{\m}{2}. 
 \]
 Therefore, we have 
 \[
 (1-C\sqrt{\eta}) \frac{\m}{2} \le\left(\frac{1-C\sqrt \eta}{1+C\sqrt{\eta}}\right)\frac{\m}{2} \le r \le \left(\frac{1+C\sqrt \eta}{1-C\sqrt{\eta}}\right)\frac{\m}{2} \le (1+C\sqrt{\eta}) \frac{\m}{2}
 \]
 At this point, we now have 
 \[
 B\left(x, (1-C\sqrt\eta)\frac{\m}{2}\right) \subset \Omega^\eps_s \subset B\left(x,(1+C\sqrt \eta)\frac{\m}{2}\right).
 \]
 This implies that 
 \[
 \vol\left(\Omega^\eps_s \operatorname{\Delta} B\left(x,\frac{\m}{2}\right)\right) \le C (\m)^3 \sqrt{\eta}. 
 \]
It remains to get a bound on the Fraenkel asymmetry of $\Omega$. 
 
 First note that the assumptions of the theorem ensure that 
 \[
 \vol(\Omega) \ge \frac{4\pi}{3} (1 - C\sqrt \eta) \left(\frac{\m}{2}\right)^3. 
 \]
 Since $\Omega \subset \Omega^\eps_s$, the conclusions of the previous paragraph imply that 
 \[
 \vol(\Omega) \le \frac{4\pi}{3}(1+C\sqrt \eta)\left(\frac{\m}{2}\right)^3. 
 \]
In particular, there is a radius $\tilde r$ satisfying 
\[
(1-C\sqrt \eta)\frac{\m}{2} \le \tilde r \le (1+C\sqrt \eta) \frac{\m}{2}
\]
and $\vol(B(x,\tilde r)) = \vol(\Omega)$.  Now observe that 
\[
\vol\left(\Omega^\eps_s \operatorname{\Delta} B(x,\tilde r)\right) \le C(\m)^3 \sqrt \eta.
\]
Since $\vol(\Omega \operatorname{\Delta} B(x,\tilde r)) \le \vol(\Omega \operatorname{\Delta} \Omega_s^\eps) + \vol(\Omega_s^\eps \operatorname{\Delta} B(x,\tilde r))$, to complete the proof, it therefore suffices to show that $\vol(\Omega \operatorname{\Delta}\Omega^\eps_s) \le C (\m)^3 \sqrt{\eta}$. This follows from 
\begin{align*}
\vol(\Omega \operatorname{\Delta} \Omega_s^\eps) &= \vol(\Omega_s) - \vol(\Omega) \\
&\le \frac{4\pi}{3}(1+C\sqrt \eta) \left(\frac{\m}{2}\right)^3 - \frac{4\pi}{3}(1-C\sqrt \eta)\left(\frac{\m}{2}\right)^3 \\
&= C (\m)^3 \sqrt\eta,
\end{align*}
and the proof of the theorem is complete.
\end{proof}

\appendix 

\section{General Asymptotically Flat Manifolds}
\label{general}

In this appendix, we discuss what can be said in the case of a general asympotically flat manifold. The extension of our results to asymptotically flat manifolds which are {\it harmonically flat at infinity} is relatively straightforward. 

\begin{defn} Let  $(M^3,g)$ be an asymptotically flat manifold with one end. We say $(M,g)$ is harmonically flat at infinity provided there is a compact set $K\subset M$ such that $M\setminus K$ is diffeomorphic to $\R^3\setminus B_r$ and, in the coordinates induced by this diffeomorphism, the metric satisfies $g = f^4 g_{\text{euc}}$ where $f$ is a harmonic function with $f\to 1$ at infinity. 
\end{defn}

\begin{prop}
\label{harmonic-flat-construction}
    Assume that $(M^3,g)$ is an asymptotically flat manifold which is harmonically flat at infinity. Further suppose that $M$ has non-negative scalar curvature, and that $\bd M$ is a minimal 2-sphere, and that $M$ contains no other compact minimal surfaces. Then there exists a function $u$ solving $\lap_3 u = 0$ in the $g$-metric with $u = 0$ on $\bd M$ and $u\to \infty$ at infinity. Moreover, with respect to the background Euclidean metric induced by the harmonically flat chart, $u$ satisfies the asymptotics $(\ref{expansion})$.
\end{prop}

\begin{proof}
The proposition essentially follows by combining a result of  Holopainen \cite[Theorem 3.19]{holopainen1991nonlinear} with the asymptotic expansion due to Kichenassamy and V\'eron \cite{kichenassamy1986singular}. We will include the details for the reader's convenience. 

Suppose $(M^3,g)$ is asymptotically flat and harmonically flat at infinity.  Then there is a diffeomorphism $\phi\colon M\setminus K \to \R^3\setminus B_r$ and under this diffeomorphism the metric takes the form $g = \phi^*(f^4 g_{\text{euc}})$ where $f\to 1$ at infinity. In fact, since $\bd M$ is a minimal 2-sphere and $M$ contains no other compact minimal surfaces, it is known that $M$ is globally diffeomorphic to $\R^3$ minus a ball \cite[Lemma 4.1]{huisken2001inverse}. Hence we can assume that $\phi$ extends to a global diffeomorphism $\phi\colon M\to \R^3 \setminus B_{s}$ for some $s < r$. Of course, the identity $g = \phi^*(f^4 g_{\text{euc}})$ holds only outside of the set $K$.

Now we proceed, closely following the presentation of Holopainen \cite[Theorem 3.19]{holopainen1991nonlinear}. Choose a sequence $r_i\to \infty$. For $\rho \ge r$, let $S_\rho = \phi\inv(\bd B_{\rho})$ and let $D_\rho = \phi\inv(\R^3\setminus B_{\rho})$. Let $w_i$ be the 3-harmonic function with respect to the $g$ metric which is 0 on $\bd M$ and $1$ on $S_{r_i}$. We can extend $w_i$ to be defined on all of $M$ by setting $w_i = 1$ on $D_{r_i}$. 
Now fix $\rho \ge 2r$. For each sufficiently large $i$, define 
\begin{gather*}
m_i(\rho) = \inf\{w_i(x): x\in \bd B_{\rho}\},\\
M_i(\rho) = \sup\{w_i(x): x\in \bd B_{\rho}\}. 
\end{gather*}
It follows from the comparison principle that $w_i \ge m_i(\rho)$ in $D_\rho$  and  $w_i \le M_i(\rho)$ on $M\setminus D_\rho$. Hence we have 
$
\{w_i \ge M_i(\rho)\} \subset D_\rho \subset \{w_i \ge m_i(\rho)\}. 
$

For a set $A\subset M$ containing the end, define the relative capacity 
\[
\cp_3(\bd M,A) = \inf\left\{\int_M \vert \grad w\vert^3\, dv: w=0 \text{ on } \bd M \text{ and } w\ge 1\text{ on } A\right\}. 
\]
Then the previous set inclusions imply the following comparison of relative capacities:
\[
\cp_3(\bd M, \{w_i\ge M_i(\rho)\}) \le \cp_3(\bd M, D_\rho) \le \cp_3(\bd M,  \{w_i\ge m_i(\rho)\}). 
\]
The Harnack inequality for $3$-harmonic functions implies that there is a constant $\lambda > 0$ independent of $i$ such that $M_i(\rho) \le \lambda m_i(\rho)$. 
Next observe that 
\[
M_i(\rho) \le \lambda m_i(\rho) = \lambda \left(\frac{\cp_3(\bd M,D_{r_i})}{\cp_3(\bd M, \{w_i\ge m_i(\rho)\})}\right)^{1/2} \le \lambda \left(\frac{\cp_3(\bd M,D_{r_i})}{\cp_3(\bd M,D_\rho)}\right)^{1/2}.
\]
Likewise, we have 
\[
m_i(\rho) \ge \lambda^{-1}M_i(\rho) = \lambda^{-1} \left(\frac{\cp_3(\bd M,D_{r_i})}{\cp_3(\bd M, \{w_i\ge M_i(\rho)\})}\right)^{1/2} \ge \lambda^{-1} \left(\frac{\cp_3(\bd M,D_{r_i})}{\cp_3(\bd M,D_\rho)}\right)^{1/2}.
\]
Therefore, it follows that the functions $u_i = \cp_3(\bd M,D_{r_i})^{-1/2}w_i$ are uniformly bounded above and below (away from 0) on $S_\rho$. 

This proves that the functions $u_i$ are locally uniformly bounded. Moreover, the H\"older continuity estimates for 3-harmonic functions imply that the family $u_i$ is equicontinuous. Therefore, after passing to a subsequence, the functions $u_i$ converge locally uniformly to a limit $u$.  This limit function $u$ is 3-harmonic with respect to the $g$-metric and satisfies $u = 0$ on $\bd M$. 

It remains to show that $u$ has the desired asymptotic expansion. First, observe that the above inequalities imply that 
\[
\lambda\inv \cp_3(\bd M,D_\rho)^{-1/2} \le u(x) \le \lambda \cp_3(\bd M,D_\rho)^{-1/2}
\]
for all $x\in S_\rho$, provided $\rho$ is sufficiently large. Next, since $\phi$ is a smooth diffeomorphism, it is in particular a $K$-quasiconformal map for some $K \ge 1$. It follows that 
\[
K\inv \cp_3(\bd B_s, \R^3\setminus B_\rho) \le \cp_3(\bd M,D_\rho) \le K \cp_3(\bd B_s, \R^3\setminus B_\rho).
\]
It is straightforward to verify that 
\[
\cp_3(\bd B_s, \R^3\setminus B_\rho) = 4\pi \left(\frac{1}{\log \rho - \log s}\right)^2. 
\]
Hence, when $\vert x\vert$ is sufficiently large, we have 
\[
C_1 \log \vert x\vert \le u(x) \le C_2 \log \vert x\vert 
\]
for some positive constants $C_1,C_2$ depending on $M$. Since $u$ is also 3-harmonic with respect to the Euclidean metric, the result of Kichenassamy and V\'eron \cite{kichenassamy1986singular} now implies that $u$ has the asymptotic expansion (\ref{expansion}). This completes the proof. 
\end{proof}

Given Proposition \ref{harmonic-flat-construction}, it is now straightforward to define the conformal capacity of an asymptotically flat manifold $(M^3,g)$ which is harmonically flat at infinity. Indeed, given the function $u$ from Proposition \ref{harmonic-flat-construction}, we let $\cc(\bd M) = e^{-a}$ where $u = \log \vert x\vert + a + o(1)$ is the constant in the asymptotic expansion of $u$. 
% Note that it is possible in principle that this conformal capacity may depend on the choice of harmonically flat coordinate chart, although we conjecture that this is not the case. 
All constructions in the paper now carry through essentially unchanged to yield an inequality between the mass and the conformal capacity. This gives the following theorem. 

\begin{theorem}
    Assume that $(M^3,g)$ is asymptotically flat and harmonically flat near infinity. Assume that $M$ has non-negative scalar curvature, and that $\bd M$ is a minimal 2-sphere, and that $M$ contains no other closed minimal surfaces.   Then $\m \ge 2\cc(\bd M)$. 
\end{theorem}

For a general asymptotically flat manifold, we are unsure whether there exists a conformal capacity function $u$ satisfying the asymptotics (\ref{expansion}), and hence we are unable to define the conformal capacity. The primary issue is that when the metric $g$ is not conformal to Euclidean near infinity, we cannot apply the result of Kichenassamy and V\'eron \cite{kichenassamy1986singular} to deduce the asymptotic expansion of $u$. In this regard, we note that Benatti, Fogagnolo, and Mazzieri \cite{benatti2024asymptotic} have recently obtained the asymptotic expansion of the $p$-capacitary potential on a general asymptotically flat manifold when $p$ is strictly between 1 and the dimension. It would be very interesting to know whether this can be adapted to the conformal case $p=3$ in dimension 3. 

Despite this difficulty, it is well-known that every asymptotically flat manifold with non-negative scalar curvature can be approximated by a manifold which is harmonically flat at infinity. Hence, for the purpose of obtaining a lower bound on the mass, the harmonically flat case is already  sufficient.

\bibliographystyle{plain}
\bibliography{bibliography.bib}

\end{document}